\documentclass[11 pt, final]{article}
\title{Continuous cocycle superrigidity for the full shift over a finitely generated torsion group.}
\author{David Bruce Cohen}

\usepackage{fullpage}
\usepackage[latin1]{inputenc}
\usepackage{tikz}
\usetikzlibrary{shapes,arrows}
\usepackage{graphicx}
\usepackage{amsmath}
\usepackage{amssymb}
\usepackage{amsthm}
\usepackage{epsfig,pinlabel}
\usepackage{amsfonts}
\usepackage{amscd}
\usepackage[all,tips]{xy}

\def\nbhd{{\mathcal{N}}}
\def\blone{{B(L,1_G)}}
\def\genset{{\mathcal{S}}}
\def\zerow{{\vec{0}}}
\def\Deo{{\Delta(\zerow)}}
\def\BtL{{B(3L,1_G)}}
\def\ngazzp{{\mathcal{N}_L(\gamma(\mathbb{Z}_{\geq 0}))}}
\def\ngazzn{{\mathcal{N}_L(\gamma(\mathbb{Z}_{\leq 0}))}}
\def\ZZp{{\mathbb{Z}_{\geq 0}}}
\def\ZZn{{\mathbb{Z}_{\leq 0}}}
\def\tilga{{\widetilde{\gamma}}}
\def\tilg{{\widetilde{g}}}
\def\olx{{\overline{x}}}
\def\tilx{{\widetilde{x}}}

\DeclareMathOperator{\ZZ}{\ensuremath{\mathbb{Z}}}
\DeclareMathOperator{\NN}{\ensuremath{\mathbb{N}}}

\DeclareMathOperator{\Cay}{Cay}

\def\To{{\rightarrow}}

\theoremstyle{plain}
\newtheorem{theorem}{Theorem}[section]
\newtheorem{lemma}[theorem]{Lemma}

\newtheorem{definition}[theorem]{Definition}

\newtheorem{proposition}[theorem]{Proposition}

\begin{document}

\maketitle

\begin{abstract}
Chung and Jiang showed that, if a one ended group contains an infinite order element, then every continuous cocycle over the full shift, taking values in a discrete group, must be cohomologous to a homomorphism. We show that their conclusion holds for all one ended groups, so that the hypothesis of admitting an infinite order element may be omitted. 
\end{abstract}

\section{Introduction.}
\label{section:intro}

Let $A$ be a finite set equipped with the discrete topology, let $G$ be a finitely generated group, and let $H$ be a discrete group. The full shift $A^G$ consists of all functions $G\To A$ equipped with the product topology and the left $G$-action given by $(gx)(h)=x(g^{-1}h)$. A result of Chung and Jiang \cite[Corollary 1]{cj} shows that when $G$ is one ended (Definition \ref{definition:ends}) and contains an infinite order element, $A^G$ has continuous cocycle rigidity---meaning that every continuous $H$-valued cocycle over $A^G$ is cohomologous to some homomorphism from $G$ to $H$, in the sense of the following definition.

\begin{definition}
(See \cite[Definitions 2.1,2.2]{cj} and Figure \ref{figure:cocycleidentity}.) Given a topological dynamical system consisting of a group $G$ acting by homeomorphisms on a space $X$, a continuous $H$-valued cocycle over this system is a continuous function $c:G\times X\to H$ such that $c(gh,x)=c(g,hx)c(h,x)$. If there exists a continuous function $b:X\To H$ and a homomorphism $\phi:G\To H$ such that $c(g,x)=b(gx)\phi(g)b(x)^{-1}$, then $c$ is said to be cohomologous to $\phi$ with transfer function $b$.
\end{definition}

\begin{figure}[t]
\labellist
\small\hair 2pt

\pinlabel $x$ at 50 65
\pinlabel $c(gh,x)$ at 175 65
\pinlabel $ghx$ at 287 65
\pinlabel $hx$ at 178 300
\pinlabel $c(h,x)$ at 78 173
\pinlabel $c(g,hx)$ at 226 173

\pinlabel $gx$ at 345 284
\pinlabel $c(g,x)$ at 465 288
\pinlabel $x$ at 585 284
\pinlabel $b(gx)$ at 368 165
\pinlabel $\phi(g)$ at 465 72
\pinlabel $b(x)$ at 550 165

\endlabellist

\centering
\centerline{\psfig{file=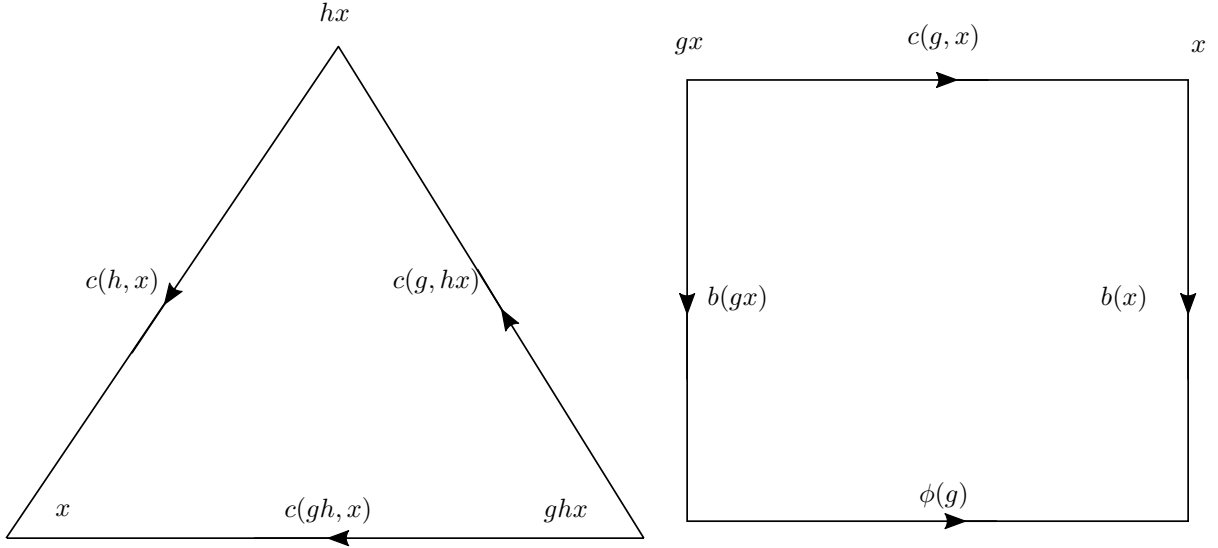,scale=80}}
\caption{At left, a van Kampen diagram depicting the cocycle identity. At right, a van Kampen diagram depicting a cohomology from $c$ to $\phi$ with transfer function $b$. Vertices are labeled so that if an edge labeled $h\in H$ connects a vertex labeled $y\in X$ to a vertex labeled $x\in X$, then there is some $g\in G$ such that $y=gx$ and $h=c(g,x)$.}
\label{figure:cocycleidentity}
\end{figure}

Continuous cocycles arise naturally in the study of continuous orbit equivalence \cite{li1} and topological couplings \cite{li2}. Our main theorem is that $A^G$ has continuous cocycle rigidity for every one ended group $G$. In particular, our theorem holds even if $G$ is periodic (has no infinite order element).

\begin{theorem}
\label{theorem:mainintro}
If $G$ is a one ended group, and $H$ a discrete group, then every continuous cocycle $c:G\times A^G\To H$ is cohomologous to a homomorphism.
\end{theorem}

See Theorem \ref{theorem:main} for the proof. Our technique loosely follows that of \cite{cj}, but we use geodesics in the Cayley graph where \cite{cj} uses paths of the form $n\mapsto a^n$ (for $a\in G$ having infinite order).

\subsection{Outline}
\label{subsection:outline}
In \S \ref{section:geometry} we recall some basic facts about the geometry of finitely generated groups. In particular, we define geodesics and one-endedness, and prove the basic fact that every group contains a biinfinite geodesic.

In \S \ref{section:cocycles} we prove Lemma \ref{lemma:neighborhood}, which states that, for a continuous cocycle $c$, and an appropriate $L>0$, the value of $c(g,x)$ is determined by $c(h,x)$ together with the values of $x$ in the $L$-neighborhood of any path from $h^{-1}$ to $g^{-1}$.

\paragraph{Proof strategy.} Inspired by \cite{cj}, in \S \ref{section:maintheorem} we use the following strategy to prove our main theorem. Suppose that $G$ is a one ended group, $H$ a discrete group, and $c:G\times A^G\To H$ a continuous cocycle.
Take $0$ to denote some element of $A$.
\begin{itemize}
\item Define $\phi(g)=c(g,\zerow)$ where $\zerow\in A^G$ is the configuration which is $0$ everywhere.
\item For $x\in \Deo$, take $b(x)$ to be $\phi(g_x)c(g_x,x)^{-1}$, where $\Deo\subseteq A^G$ consists of configurations with only finitely many non-$0$ entries, and $g_x\in G$ is chosen so that its word norm is much larger than that of any $g$ such that $x(g)\neq 0$ (see Definition \ref{definition:phi}).
\item Show that the restriction of $c$ to $G\times\Deo$ is cohomologous to $\phi$ with transfer function $b$ (Proposition \ref{proposition:coboundary}).
\item Show that $b(x)$ only depends on a bounded set of coordinates of $x\in\Deo$, and hence $b$ may be extended continuously to all of $A^G$ (Proposition \ref{proposition:agreement}).
\item Show that $c$ is cohomologous to $\phi$ with transfer function $b$ (Theorem \ref{theorem:main}).
\end{itemize}

\paragraph{Acknowledgments.} The author is grateful for the support of NSF grant 1502608. He also wishes to thank Yongle Jiang for discussing his work \cite{cj}, and Andy Putman for some helpful suggestions.

\section{Geometry of groups}
\label{section:geometry}
We now briefly survey some relevant basic facts about the geometry of finitely generated groups. In particular, we recall the definitions of number of ends of a group and geodesics in a group, show that biinfinite geodesics exist in any infinite group (Lemma \ref{lemma:geodesicline}), and show that for $L>0$ and $\gamma$ a geodesic, the $L$-neighborhood of $\gamma(\{0,1,2\ldots\})$ intersects the $L$-neighborhood of $\gamma(\{0,-1,-2,\ldots\})$ in a set contained inside the ball of radius $3L$ around $\gamma(0)$ (Lemma \ref{lemma:intersection}).

\paragraph{Notation, assumptions, and conventions.} For the rest of this paper, $\genset$ will denote a finite, symmetric (i.e., $\genset=\genset^{-1})$ generating set for the infinite group $G$. Denote the identity element of $G$ by $1_G$. Every element of $G$ is represented by some word in $\genset$, and the word norm $|g|$ of any $g\in G$ is defined to be the length of the shortest word in $\genset$ representing $g$. That is,
$$|g|:=\inf\{\ell:g=s_1\ldots s_\ell;s_1,\ldots,s_\ell\in\genset\}.$$

For $g,h\in G$, define the left-invariant distance
$$d(g,h):=|g^{-1}h|,$$
and for any $T\subseteq G$, let $\nbhd_L T:=\{g\in G:\inf\{d(g,s):s\in T\}\leq L\}$ denote the $L$-neighborhood of $T$. Let $B(r,g)$ denote the closed $r$-ball around $g\in G$, i.e., $B(r,g):=\nbhd_r\{g\}$. Many of our figures are van Kampen diagrams---i.e., they depict planar graphs, with directed edges labeled by group elements, such that every loop represents the identity. 

\begin{definition}
\label{definition:geodesic}
Let $I$ be an interval in $\ZZ$---i.e., the intersection of $\ZZ$ with some possibly unbounded real interval. A function $\gamma:I\To G$ is said to be a geodesic if $d(\gamma(m),\gamma(n))=|n-m|$ for all $n,m\in I$.

We say that $\gamma:I\To G$ is a path if $d(\gamma(j-1),\gamma(j))=1$ whenever $j-1$ and $j$ are in $I$.
\end{definition}

Obviously, this is equivalent to considering paths and geodesics in the Cayley graph of $G$. For example, if $\ZZ^2$ is equipped with the standard generating set, then the map $\gamma:\ZZ\To\ZZ^2$ given by $\gamma:n\mapsto (n,0)$ is a geodesic.

\begin{lemma}
\label{lemma:geodesicsegment}
If $g\in G$, there exists a geodesic $\gamma:\{0,\ldots,|g|\}\To G$ such that $\gamma(0)=1_G$ and $\gamma(|g|)=g$.
\end{lemma}

\begin{proof}
Let $s_1\ldots s_\ell$ be a minimal length word in $\genset^*$ representing $g$. By definition of the word metric, $\ell=|g|$. For $j=1,\ldots,\ell$, set $\gamma(j)$ to be the element of $g$ represented by the length $j$ prefix $s_1\ldots s_j$ of $s_1\ldots s_\ell$, and set $\gamma(0)=1_G$. Because $s_1\ldots s_\ell$ has minimal length among words representing $g$, $s_{j+1}\ldots s_k$ has minimal length among words representing $\gamma(j)^{-1}\gamma(k)$ for any $0\leq j<k\leq \ell$. In particular, $|\gamma(j)^{-1}\gamma(k)|=k-j$ for all such $j,k$---i.e., $\gamma$ is a geodesic.
\end{proof}

Recall our assumption that $G$ is infinite.

\begin{lemma}
\label{lemma:geodesicline}
There exists a geodesic $\gamma:\ZZ\To G$.
\end{lemma}

\begin{proof}
Because $G$ is infinite, the word norm on $G$ achieves every natural number. For each $r\in\NN$, choose some $g_r\in G$ such that $|g_r|=2r$. By Lemma \ref{lemma:geodesicsegment}, there is a geodesic $\gamma_r:\{0,\ldots,2r\}\To G$ such that $\gamma_r(0)=1_G$ and $\gamma_r(2r)=g_r$. Let $\tilga_r:\{-r,\ldots,r\}\To G$ be defined by $\tilga_r(n):=\gamma_r(r)^{-1}\gamma_r(n+r)$, so that $\tilga_r$ is a geodesic and $\tilga_r(0)=1_G$. By a trivial case of Arzela-Ascoli, the sequence $(\tilga_r)_{r\in\NN}$ subconverges pointwise to a geodesic $\gamma:\ZZ\To G$.
\end{proof}

Let $\ZZp=\{n\in\ZZ:n\geq 0\}$ and $\ZZn:=\{n\in\ZZ:n\leq 0\}$. Recall that $\nbhd_L(T)$ denotes the $L$-neighborhood of $T\subseteq G$. The following lemma bounds the intersection between the $L$-neighborhoods of the left and right halves of a geodesic $\gamma:\ZZ\To G$. It may be compared to \cite[Lemma 3.3]{cj}, which proves an analogous result for a path of the form $n\mapsto a^n$ (where $a$ has infinite order).

\begin{lemma}
\label{lemma:intersection}
If $\gamma:\ZZ\To G$ is a geodesic, then
$$\nbhd_L(\gamma(\ZZp))\cap\nbhd_L(\gamma(\ZZn))\subseteq B(3L,1_G).$$
\end{lemma}

(See Figure \ref{figure:intersection}).

\begin{proof}
Suppose that $g\in\nbhd_L(\gamma(\ZZp))\cap\nbhd_L(\gamma(\ZZn))$. Then there exists $n_+\geq 0$ and $n_-\leq 0$ such that $d(g,\gamma(n_+))\leq L$ and $d(g,\gamma(n_-))\leq L$. Because $\gamma$ is a geodesic,
$$|n_+-n_-|=d(\gamma(n_+),\gamma(n_-))
\leq d(\gamma(n_+),g) + d(g,\gamma(n_-))\leq 2L,
$$
so $n_+\leq n_-+|n_+-n_-|\leq 2L$. Thus
$$|g|=d(g,1_G)\leq d(g,\gamma(n_+))+d(\gamma(n_+),1_G)\leq L+2L=3L$$
\end{proof}

\begin{figure}[t]
\labellist
\small\hair 2pt

\pinlabel $\gamma(-2L)$ at 5 57
\pinlabel $\gamma(0)$ at 300 74
\pinlabel $\gamma(2L)$ at 600 89
\pinlabel $\leq L$ at 310 153
\pinlabel $\leq L$ at 460 159
\pinlabel $g$ at 380 206
\pinlabel $\gamma(n_-)$ at 243 74
\pinlabel $\gamma(n_+)$ at 522 91

\endlabellist

\centering
\centerline{\psfig{file=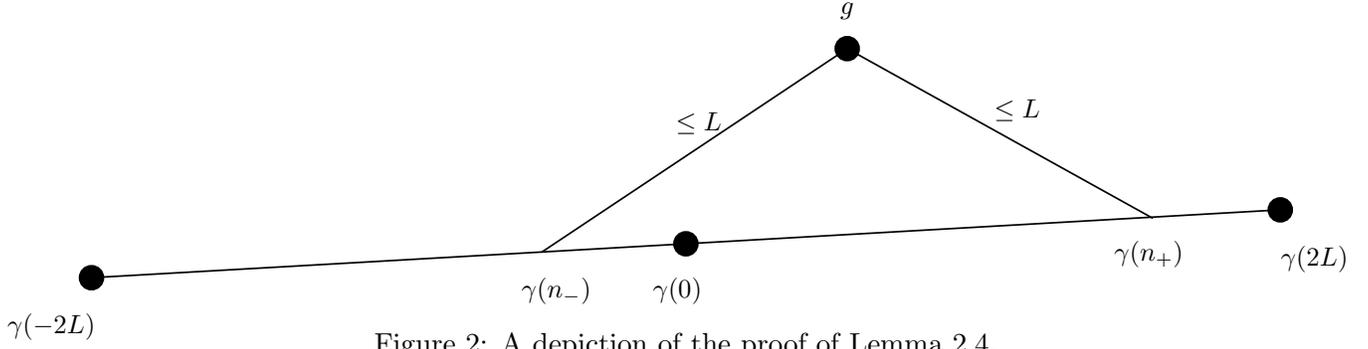,scale=80}}
\caption{A depiction of the proof of Lemma \ref{lemma:intersection}.}
\label{figure:intersection}
\end{figure}

\paragraph{Ends of groups.} Let $\Cay_\genset G$ denote the Cayley graph of $G$ with respect to $\genset$.
\begin{definition}
\label{definition:ends}
The number of ends of $G$ is the limit, as $n\To\infty$, of the number of unbounded connected components of $\Cay_\genset G\setminus B(n,1_G)$.
\end{definition}

\paragraph{Assumption: one endedness.} For the rest of the paper, we will assume that $G$ is one-ended. This implies that for any $r$, $\Cay_\genset G\setminus B(r,1_G)$ has exactly one unbounded connected component.
\begin{definition}
\label{definition:N}
For $r\in\NN$, let
$$N(r):=\sup\{|g|:g\text{ does not belong to the unbounded component of }\Cay_\genset G\setminus B(r,1_G)\}.$$
\end{definition}

Note that $N(r)<\infty$, as, by bounded degree of $\Cay_\genset G$, the supremum is taken over a finite set---the number of bounded components is finite since each is adjacent to $B(r,1_G)$ in $\Cay_\genset G$.

\section{Cocycles over the full shift.}
\label{section:cocycles}
Recall that, by standing assumption, $G$ is a one ended group. For the rest of the paper, fix a continuous cocycle $c:G\times A^G\to H$, where $H$ is a discrete group. 
 That $c$ is continuous means that, for every $g\in G$, there is some finite subset $F\subseteq G$ (depending on $g$) such that $c(g,x)$ depends only on the restriction $x|_{F}$ of $x$ to $F$. That is, if $x|_{F}=y|_{F}$ then $c(g,x)=c(g,y)$.

\begin{definition}
\label{definition:L}
Let $L\in\NN$ be such that, for any generator $s\in\genset$, the function $c(s,x)$ is determined by $x|_{B(L,1_G)}$.
\end{definition}

Recall that $0$ denotes some element of $A$.
\begin{definition}
\label{definition:zero}
Let $\zerow\in A^G$ be the all-$0$ configuration, i.e., $\zerow(g)=0$ for all $g\in G$. Let $\Deo\subseteq A^G$ consist of all $x\in A^G$ such that $x(g)=0$ for all but finitely many $g\in G$, and for $x\in\Deo$, let $\|x\|$ denote
$\sup\{|g|:x(g)\neq 0\}.$
\end{definition}

Recall our definition of ``path" from Definition \ref{definition:geodesic}.
\begin{lemma}
\label{lemma:neighborhood}
Let $\gamma:\{0,\ldots,n\}\To G$ be path, and take $x\in A^G$. Then $c(\gamma(n)^{-1},x)c(\gamma(0)^{-1},x)^{-1}$ is determined by $x|_{\nbhd_L(\gamma\{0,\ldots,n\})}.$ That is, if 
$$x|_{\nbhd_L(\gamma\{0,\ldots,n\})}=y|_{\nbhd_L(\gamma\{0,\ldots,n\})},$$
then
$$c(\gamma(n)^{-1},x)c(\gamma(0)^{-1},x)^{-1}=c(\gamma(n)^{-1},y)c(\gamma(0)^{-1},y)^{-1}.$$
\end{lemma}
(See Figure \ref{figure:neighborhood}).
\begin{proof}
Let $s_j$ denote $\gamma(j-1)^{-1}\gamma(j)$ for $j=1,\ldots,n$. Since $\gamma$ is a path, we have that $s_j\in \genset$.
From the cocycle identity, we see that
$$c(\gamma(j)^{-1},x)=c(s_j^{-1},\gamma(j-1)^{-1}x)
c(\gamma(j-1)^{-1},x).$$
 Expanding recursively, we obtain
$$c(\gamma(n)^{-1},x)$$
$$=c(s_n^{-1},\gamma(n-1)^{-1}x)
c(s_{n-1}^{-1},\gamma(n-2)^{-1}x)$$
$$\ldots
c(s_1^{-1},\gamma(0)^{-1}x)
c(\gamma(0)^{-1},x).$$
Now, $c(s_j^{-1},\gamma(j-1)^{-1}x)$ is determined by $(\gamma(j-1)^{-1}x)|_\blone$, which in turn is determined by $x|_{B(L,\gamma(j-1))}$ because $(\gamma(j-1)^{-1}x)(g)=x(\gamma(j-1)g)$.
\end{proof}

\begin{figure}[t]
\labellist
\small\hair 2pt

\pinlabel $x$ at 3 170
\pinlabel $\gamma(j-1)^{-1}x$ at 234 222
\pinlabel $\gamma(j)^{-1}x$ at 228 94

\pinlabel $x$ at 300 170
\pinlabel $\gamma(0)^{-1}x$ at 544 312
\pinlabel $\gamma(n)^{-1}x$ at 544 10

\pinlabel $c(\gamma(j-1)^{-1},x)$ at 114 200
\pinlabel $c(\gamma(j)^{-1},x)$ at 114 100
\pinlabel $c(s_j^{-1},\gamma(j-1)^{-1}x)$ at 148 150

\pinlabel $c(\gamma(0)^{-1},x)$ at 400 250
\pinlabel $c(\gamma(1)^{-1},x)$ at 410 174
\pinlabel $c(\gamma(n-1)^{-1},x)$ at 410 140
\pinlabel $c(\gamma(n)^{-1},x)$ at 400 68

\pinlabel $c(s_1^{-1},\gamma(0)^{-1}x)$ at 572 270
\pinlabel $\gamma(1)^{-1}x$ at 554 236
\pinlabel $\ldots$ at 554 150
\pinlabel $\gamma(n-1)^{-1}x$ at 564 80
\pinlabel $c(s_n^{-1},\gamma(n-1)^{-1}x)$ at 584 48

\endlabellist

\centering
\centerline{\psfig{file=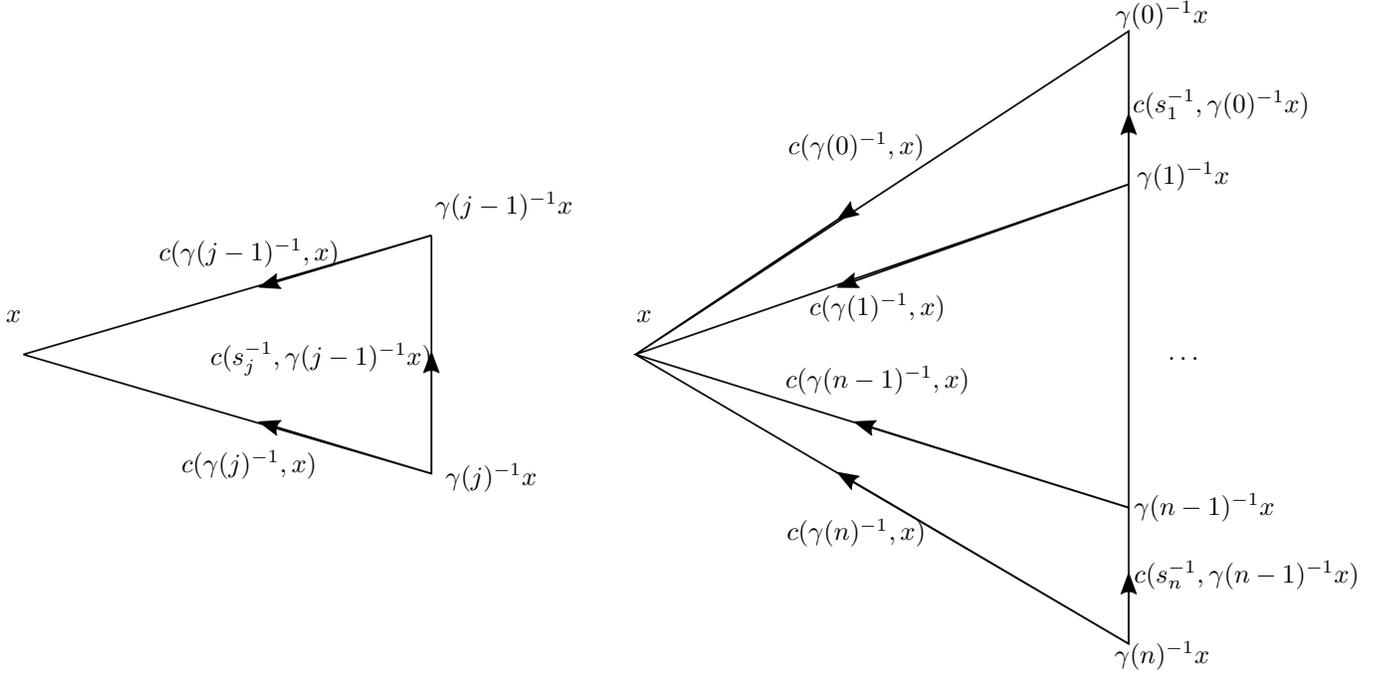,scale=80}}
\caption{A pair of van Kampen diagrams depicting the proof of Lemma \ref{lemma:neighborhood}.}
\label{figure:neighborhood}
\end{figure}

\section{Proof of the main theorem.}
\label{section:maintheorem}
Recall that, by standing assumption $G$ is a one ended group, and recall the notations $\zerow,\Deo$ and $\|x\|$ from Definition \ref{definition:zero}. We will now show that the continuous cocycle $c:G\times A^G\To H$ is cohomologous to a homomorphism---i.e., that there exists a homomorphism $\phi:G\To H$ and a continuous function $b:A^G\To H$ such that $c(g,x)=b(gx)\phi(g)b(x)^{-1}$. As a first step towards finding such $b$ and $\phi$, observe that if we had such $b$ and $\phi$, then we would have $c(g,\zerow)=b(\zerow)\phi(g)b(\zerow)^{-1}$ for every $g\in G$. For any $x\in\Deo$ and $|g|$ sufficiently large depending on $x$, we would have $b(gx)=b(\zerow)$ by continuity of $b$, and hence $c(g,x)=b(\zerow)\phi(g)b(x)^{-1}$. Equivalently, $b(x)$ would equal $b(\zerow)\phi(g)c(g,x)^{-1}$. Thus, the following definition of $\phi$ and $b|_\Deo$ corresponds to the normalization $b(\zerow)=1_H$. 

\begin{definition}
\label{definition:phi}
Let $\phi(g):=c(g,\zerow)$. For each $x\in\Deo$, fix some $g_x\in G$ such that $|g_x|>N(\|x\|+L)$ and set $b(x):=c(g_x,x)^{-1}\phi(g_x)$.
\end{definition}

Observe that $\phi$ is a homomorphism by the cocycle identity, since $\zerow$ is a fixed point for the action of $G$ on $A^G$. We will show that the restriction of $c$ to $G\times \Deo$ is cohomologous to $\phi$ with transfer function $b$, and then we will show that $b$ extends continuously to all of $A^G$. First, we see that the choice of $g_x$ in Definition \ref{definition:phi} is irrelevant.

\begin{lemma}
\label{lemma:independence}
If $|g|>N(\|x\|+L)$ then $c(g,x)^{-1}\phi(g)=b(x)$.
\end{lemma}

\begin{proof}
Let $\gamma:\{0,\ldots,n\}$ be a path such that $\gamma$ connects $g^{-1}$ to $g_x^{-1}$ outside of $B(\|x\|+L,1_G)$ (so that $\gamma(0)^{-1}=g$ and $\gamma(n)^{-1}=g_x$ and $|\gamma(j)|>\|x\|+L$ for $j=0,\ldots, n$). Such a path exists because $|g|$ and $|g_x|$ are greater than $N(\|x\|+L)$. Since $|\gamma(j)|>\|x\|+L$ for all $j\in \{0,\ldots,n\}$ and $x(g)=0$ whenever $|g|>\|x\|$, we have that $x(g)=0$ for any $g\in\nbhd_L(\gamma\{0,\ldots,n\})$, so by Lemma \ref{lemma:neighborhood}, we have

$$c(\gamma(n)^{-1},x)c(\gamma(0)^{-1},x)^{-1}
=c(\gamma(n)^{-1},\zerow)c(\gamma(0)^{-1},\zerow)^{-1}
.$$
Consequently
$$c(g_x,x)c(g,x)^{-1}=\phi(g_x g^{-1}),$$
or equivalently
$$c(g,x)^{-1}\phi(g)=c(g_x,x)^{-1}\phi(g_x)=b(x)$$
as desired.
\end{proof}

We now show that the restriction of $c$ to $G\times\Deo$ is cohomologous to $\phi$ with transfer function $b$.

\begin{proposition}
\label{proposition:coboundary}
For all $g\in G$ and $x\in\Deo$, we have $c(g,x)=b(gx)\phi(g)b(x)^{-1}$.
\end{proposition}

(See Figure \ref{figure:coboundary}).

\begin{proof}
Choose some $\tilg\in G$ such that $|\tilg|>N(\|x\|+L)$ and $|\tilg|>|g|+N(\|gx\|+L)$.
Then we have $|\tilg g^{-1}|\geq |\tilg|-|g|>N(\|gx\|+L)$, so by Lemma \ref{lemma:independence} we have $b(gx)=c(\tilg g^{-1},gx)^{-1}\phi(\tilg g^{-1})$ and $b(x)=c(\tilg,x)^{-1}\phi(\tilg)$.
It follows that $c(\tilg g^{-1},gx)^{-1}=b(gx)\phi(\tilg g^{-1})^{-1}$ and $c(\tilg,x)=\phi(\tilg)b(x)^{-1}$.

By the cocycle identity,
$$c(\tilg, x)=c(\tilg g^{-1},gx)c(g,x).$$
Rearranging, we obtain
$$c(g, x)=c(\tilg g^{-1},gx)^{-1}c(\tilg,x)
=(b(gx)\phi(\tilg g^{-1})^{-1})(\phi(\tilg)b(x)^{-1})
$$
$$=b(gx)\phi(g)b(x)^{-1}$$
as desired.
\end{proof}

\begin{figure}[t]
\labellist
\small\hair 2pt

\pinlabel $x$ at 0 156
\pinlabel $\tilg x$ at 290 45
\pinlabel $gx$ at 290 284

\pinlabel $c(g,x)$ at 170 324
\pinlabel $c(\tilg,x)$ at 170 38
\pinlabel $c(\tilg g^{-1},gx)$ at 354 156

\pinlabel $\phi(g)$ at 170 190
\pinlabel $\phi(\tilg)$ at 170 126
\pinlabel $\phi(\tilg g^{-1})$ at 190 156

\pinlabel $b(x)$ at 80 180
\pinlabel $b(gx)$ at 220 240
\pinlabel $1_G$ at 216 80

\pinlabel $A$ at 130 240
\pinlabel $B$ at 264 156
\pinlabel $C$ at 130 74

\endlabellist

\centering
\centerline{\psfig{file=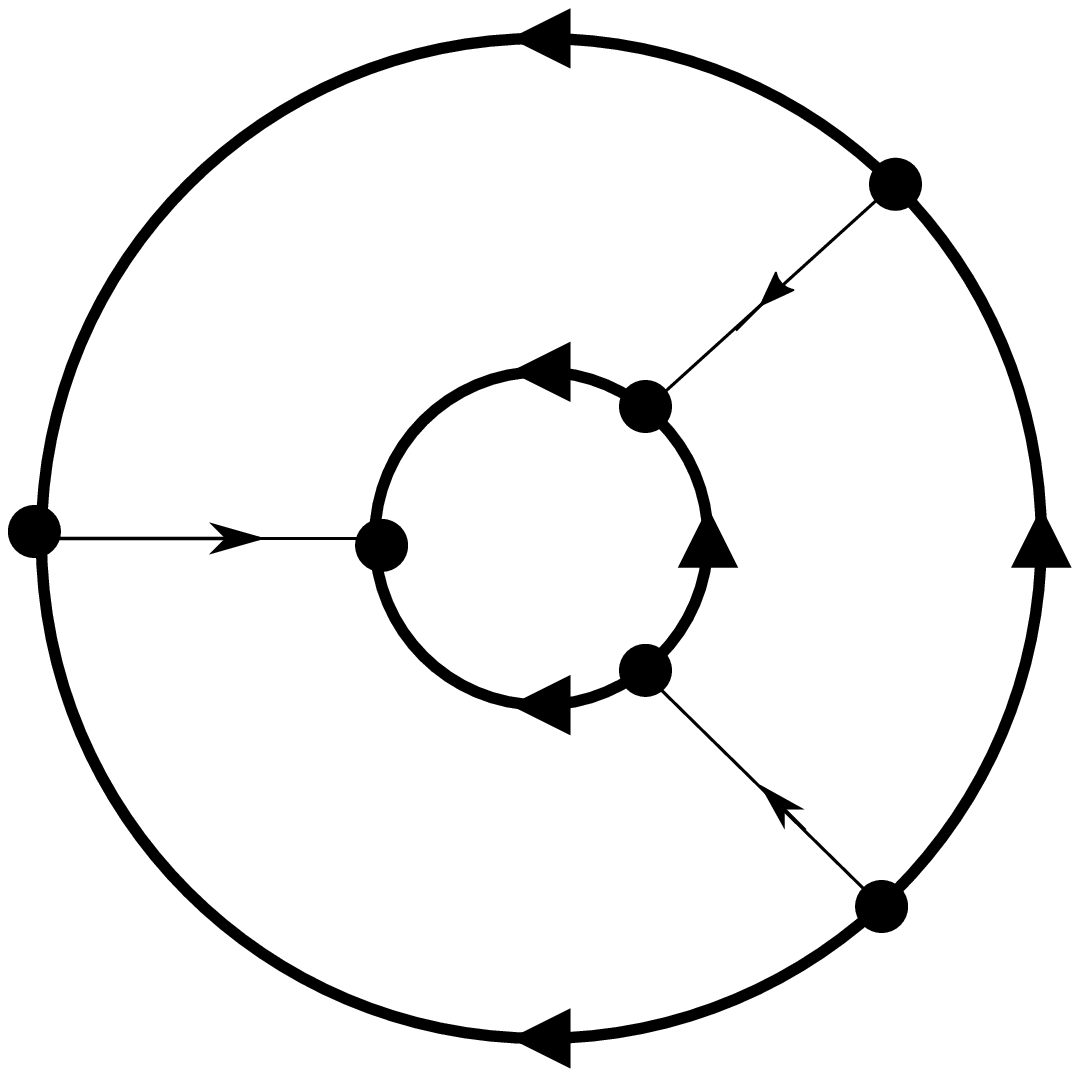,scale=80}}
\caption{A van Kampen diagram depicting the proof of Proposition \ref{proposition:coboundary}. The inner and outer circles are relations by the cocycle identity. B and C are relations by Lemma \ref{lemma:independence}. It follows that $A$ is a relation.}
\label{figure:coboundary}
\end{figure}

As $\Deo$ is noncompact, it is not even obvious a priori that $b$ takes on only finitely many values. However, we now show that $x\mapsto b(x)$ factors through the restriction $x\mapsto x|_\BtL$, so that $b$ may be extended continuously to a function $A^G\To H$. The following proposition should be compared to \cite[Lemmas 3.5,4.1]{cj}

\begin{proposition}
\label{proposition:agreement}
If $x,y\in\Deo$ with $x|_\BtL=y|_\BtL$ then $b(x)=b(y)$.
\end{proposition}
\begin{proof}
By Lemma \ref{lemma:geodesicline} there exists a geodesic $\gamma:\ZZ\To G$. We will begin by finding 
$z\in\Deo$ such that $z|_\ngazzp=x|_\ngazzp$ and $z|_\ngazzn=y|_\ngazzn$.

Let $z(g):=0$ for $g\notin \ngazzp\cup\ngazzn$, and take $z(g):=x(g)$ for $g\in\ngazzp$ and $z(g):=y(g)$ for $g\in\ngazzn$. For $g\in\ngazzn\cap\ngazzp\subseteq\BtL$ (recall Lemma \ref{lemma:intersection}), these definitions agree by our assumption that $x$ and $y$ agree on $\BtL$. Because $x,y\in\Deo$, we also have $z\in\Deo$.

By Lemma \ref{lemma:neighborhood}, $c(\gamma(n)^{-1},x)=c(\gamma(n)^{-1},z)$ for all $n\geq 0$. By Lemma \ref{lemma:independence}, this implies that $b(x)=b(z)$. Arguing similarly with $\gamma(-n)$ in place of $\gamma(n)$, we see that $b(y)=b(z)$. It follows that $b(x)=b(y)$.
\end{proof}

We will now prove our main theorem.

\begin{theorem}
\label{theorem:main}
Every continuous cocycle $c:G\times A^G\To H$ is cohomologous to a homomorphism.
\end{theorem}

\begin{proof}
Take $b$ and $\phi$ as given in Definition \ref{definition:phi}. We will first note that $b$ extends continuously to all of $A^G$, and then we will show that this extension is a transfer function for a cohomology from $c$ to $\phi$. 

\paragraph{Extending $b$ over $A^G$.}
For $x\in A^G$, let $\overline{x}\in\Deo$ be given by setting $\overline{x}(g):=x(g)$ for $g\in\BtL$ and $\overline{x}(g):=0$ for $g\notin \BtL$.  For arbitrary $x\in A^G$ we may now define $b(x):=b(\overline{x})$. By definition $b$ is continuous, and by Proposition \ref{proposition:agreement}, we have that $b:A^G\To H$ agrees with Definition \ref{definition:phi} on $\Deo$. We now proceed to show that $c(g,x)=b(gx)\phi(g)b(x)^{-1}$ for every $g\in G$ and $x\in A^G$.

\paragraph{Showing that $c$ is cohomologous to $\phi$ with transfer function $b$.} Given $g\in G$ and $x\in A^G$, we shall apply Proposition \ref{proposition:coboundary} to some $\tilx\in\Deo$ approximating $x$. Let $\tilx\in\Deo$ be given by setting $\tilx(g'):=x(g')$ for $g'\in B(|g|+3L,1_G)$ and $\tilx(g'):=0$ for $g'\notin B(|g|+3L,1_G)$. By Lemma \ref{lemma:geodesicsegment}, there is a geodesic $\gamma:\{0,\ldots,n\}\To G$ such that $\gamma(0)=1_G$ and $\gamma(n)=g^{-1}$. For such a geodesic, $\nbhd_L(\gamma\{0,\ldots,n\})\subseteq B(|g|+3L,1_G)$. It follows by Lemma \ref{lemma:neighborhood} that $c(g,x)c(1_G,x)^{-1}=c(g,\tilx)c(1_G,\tilx)^{-1},$ or in other words
 $$c(g,x)=c(g,\tilx).$$ 

We have $c(g,\tilx)=b(g\tilx)\phi(g)b(\tilx)^{-1}$ by Proposition \ref{proposition:coboundary}, as $\tilx\in\Deo$.

We have $b(\tilx)=b(\olx)=b(x)$ by definition. We see that
$$(g\tilx)|_\BtL=(gx)|_\BtL$$
because $g^{-1}\BtL\subseteq B(|g|+3L,1_G)$. Thus by Proposition \ref{proposition:agreement}, $b(g\tilx)=b(gx)$. It follows that
$$c(g,x)=c(g,\tilx)=b(g\tilx)\phi(g)b(\tilx)^{-1}=b(gx)\phi(g)b(x)^{-1},$$
as desired.
\end{proof}

\bibliographystyle{plain}
\bibliography{bibliography}

\end{document}